\documentclass[11pt,reqno]{amsart}
\usepackage[a4paper]{geometry}
\usepackage[utf8]{inputenc}

\usepackage[english]{babel}
\usepackage{amsmath,amssymb,amsfonts,amsthm}
\usepackage{mathrsfs}
\usepackage{colonequals} 

\usepackage{graphicx}
\usepackage{layout}

\usepackage[dvipsnames]{xcolor}
\colorlet{LightRubineRed}{RubineRed!70!}
\colorlet{Mycolor1}{green!10!orange!90!}
\definecolor{Mycolor2}{HTML}{00F9DE}

\usepackage{multicol}
\usepackage{enumitem}

\usepackage{amssymb}
\usepackage{fancybox}
\usepackage{fancyhdr}
\usepackage{times}

\usepackage{verbatim}
\usepackage[pdftex,bookmarks,colorlinks,linkcolor=blue]{hyperref}

\usepackage{setspace}
\usepackage{lmodern}
\usepackage{colortbl}

\newcommand{\leqs}{\leqslant}
\newcommand{\geqs}{\geqslant}

\numberwithin{equation}{section}

\newcommand{\D}{\mathbb{D}}

\newcommand{\cD}{{\mathcal{D}}}
\newcommand{\cS}{{\mathcal{S}}}
\newcommand{\cR}{{\mathcal{R}}}

\newcommand{\cA}{{\mathcal{A}}}

\newcommand{\cK}{{\bf \mathcal{K}}}

\newcommand{\beqa}{\begin{eqnarray*}}
	\newcommand{\bea}{\begin{eqnarray}}
	\newcommand{\eeqa}{\end{eqnarray*}}
	\newcommand{\eea}{\end{eqnarray}}

	\newtheorem{lem}{Lemma}[section]
	\newtheorem{cor}{Corollary}[section]
	\newtheorem{theo}{Theorem}[section]

	\title[Integration operators on Hardy and Bergman spaces]{Integration operators on Hardy and Bergman spaces}
	\date{\today}
	\author[O. El Fallah, F. Mkadmi, Y. Omari]{O. El-Fallah, F. Mkadmi, Y. Omari}
	\address{Laboratory of Mathematic Analysis and Applications, Faculty of Sciences, Mohammed V University in Rabat, 4 Av. Ibn Battouta, Morocco.}
	\email{omar.elfallah@fsr.um5.ac.ma, fouzia.mkadmii@gmail.com, omariysf@gmail.com}
	
	\begin{document}
		\maketitle
		\begin{abstract}    
			In the present work, we are interested in compact integration operators
			$I_g f(z) = \int _0^z f(\zeta)g'(\zeta)d\zeta$ acting on the Hardy space $H^2$ and on the weighted Bergman spaces $\cA^2_\alpha$. We give upper and lower estimates for the singular values of $I_g$.
		\end{abstract}
		\maketitle
		
		\section{Introduction}
		Let $\mathbb{D}$ be the open unit disc, and let $\mathbb{T} \colonequals  \partial \mathbb{D}$  be the unit circle. The Hardy space $H^2$ consists  of analytic functions $f$ on $\mathbb{D}$ such that 
		\[
		\| f\| _{H^2} \colonequals\left(\sup_{0<r<1}\int_0^{2\pi} |f(re^{i\theta})|^2  \dfrac{d\theta}{2\pi}\right)^{1/2}<\infty.
		\]
		For an analytic function $g$ on $\mathbb{D}$, we define the linear transformation 
		\[
		I_gf(z) =\displaystyle \int _0^z f(\zeta)g'(\zeta)d\zeta,\quad f\in H^2.
		\]
		These operators have been introduced by Pommerenke \cite{Pom} who proved that $I_g$ is bounded on $H^2$ if and only if $g$ is of bounded mean oscillation, that is, $g \in BMOA$. The compactness and membership to Schatten classes have been studied by Aleman and Siskakis in \cite{AS1}. They proved that $I_g$ is compact if and only if $g$ is in the space of analytic functions of vanishing mean oscillation  $VMOA$, and that $I_g $ belongs to the Schatten class $\mathcal{S}_p(H^2)$, for $p>1$, if and only if $g$ is in the Besov space $B_p$ given by
		\[
		B_p=\left  \{ g \in H^2:\ \displaystyle \int _\mathbb{D}\left ( ( 1-|z|^2)|g'(z)| \right )^p\frac{dm(z)}{(1-|z|^2)^2} \right \},
		\]
		where $m$ denotes the normalized Lebesgue measure on $\mathbb{D}$. They also proved that if ${I_g} \in \mathcal{S}_1$ then $g'=0$.\\

		The weighted Bergman space $\cA^2_\alpha$, with  $\alpha>-1$, is the space of analytic functions $f$ on $\mathbb{D}$ for which
		\[
		\|f\|_\alpha \colonequals \left((\alpha+1)\displaystyle \int_\mathbb{D} |f(z)|^2(1-|z|^2)^\alpha dm(z)\right)^{1/2} <\infty.
		\]
		In \cite{AS2}, Aleman and  Siskakis  considered  the operator $I_g$ acting on weighted Bergman spaces. They proved that $I_g$  is bounded on $\mathcal{A}^2_\alpha$ (respectively compact) if and only if $g$ is in the Bloch space $\mathcal{B}$ (respectively in the little Bloch space $\mathcal{B}_0$), and that  $I_g$ is in the class $\mathcal{S}_p(\mathcal{A}^2_\alpha)$, for $p > 1$, if and only if $g$ is in the space $B_p$. For further information about  
		the spaces used in this paper, we refer the reader to the monograph  \cite{Zhu}.\\
		
		This paper is devoted to the study of the asymptotic behavior of the singular values of compact integration operators acting on  $H^2$ or on  $\cA^2_\alpha$.  
		Before stating our main results, we give some basic notations. 
		\\
		\\
		The Carleson window (or box) associated with an arc $I\subset \mathbb{T}$ is defined by 
		\[
		W(I)=\left\lbrace z\in\mathbb{D}:1-|z|\leqs | I | /2\pi,\ z/|z|\in I\right\rbrace,
		\]
		where $|I|$ denotes the length of $I$. 
		Let $R(I)$ denote the  inner half of $W(I)$ which is given by  
		\[
		R(I)= \left\{ z\in \mathbb{D}:\ |I|/4\pi <1-|z|\leqs  | I | /2\pi ,z/|z| \in I \right\}.
		\]
		The family  $\left\{ R(I_{n,k})\right\} _{n,k}$ constitutes  a pairwise disjoint covering of $\mathbb{D}$, where
		\[
		I_{n,k} \colonequals \lbrace e^{i\theta}:2\pi k/2^n\leqs \theta <2\pi(k+1)/2^n\rbrace,\quad n\geqs 1\ \mbox{and}\ 1\leqs k\leqs 2^n,
		\]
		are the dyadic arcs.\\ \\
		The {Bloch} space $\mathcal{B}$ (respectively the little {Bloch} space $\mathcal{B}_0$) consists of analytic functions $g$ on $\mathbb{D}$ such that
		\[
		\displaystyle \sup _{z\in \mathbb{D}}(1-|z|^2)|g'(z)| <\infty \quad (\text{respectively } \displaystyle \lim _{|z| \to 1}(1-|z|^2)|g'(z)|=0).
		\]
		{Let $dm_g(z) = |g'(z)|^2 (1-|z|^2)^2dm(z)$}. It is established in \cite{AS2} that $g$ is in the space $\mathcal{B}$ if and only if
		\[ 
		\frac{m_g(R(I_{n,k}))}{{m(R(I_{n,k}))}}=O(1),
		\]
		and that $g$ is in the space $\mathcal{B}_0$ if and only if 
		\begin{equation}\label{00}
		\displaystyle \lim_{n\to \infty} \sup _{1\leqs k \leqs 2^n}\frac{m_g(R(I_{n,k}))}{{m(R(I_{n,k}))}}=0.
		\end{equation}
		Hence, $I_g$ is compact on $\cA^2_\alpha$ if and only if  $g$ satisfies assumption (\ref{00}). In this case, {let $\left(I_{n}\right)_{n\geqs 1}$ be an enumeration of $\left(I_{n,k}\right)_{n,k}$ so that the sequence 
			\[a_n(m_g) \colonequals \frac{m_g(R(I_{n}))}{{m(R(I_{n}))}}
			\] 
			is  nonincreasing.}
		\\
		\\
		Throughout this paper, the notation $A\lesssim B$ stands for $A\leqs cB$ with $c$ is a certain positive constant, and the notation $A\asymp B$ is used if both $A\lesssim B$ and $B\lesssim A$ hold.
		\\
		
		First, we state the main result of this paper in the case of  Bergman spaces.
		\begin{theo}\label{Bergman}
			Let $g\in \mathcal{B}_0$. There exists an absolute integer $p$ such that
			\begin{eqnarray}
			\sqrt{a_{pn}(m_g)} \lesssim s_{n}(I_g)\lesssim \sqrt{a_{[\frac{n}{p}]}(m_g)},\quad n\geqs 1,
			\end{eqnarray}
			where  $[\frac{n}{p}]$ is the integer part of $\frac{n}{p}$.
		\end{theo}
		
		An immediate consequence of Theorem \ref{Bergman} is given as follows.
		\begin{cor}\label{Cor_Bergman}
			Let $g\in \mathcal{B}_0$. If $\eta=(\eta(n))_n$ is a nonincreasing sequence of positive numbers such that  $\eta(2n)\asymp \eta(n)$, then
			\[
			s_n^2(I_g) \asymp \eta(n) \Longleftrightarrow a_n(m_g) \asymp \eta(n).
			\]
		\end{cor}
		
		The situation in the Hardy space is more subtle. Indeed, the boundedness of $I_g$ on $H^2$ is equivalent to the membership of the symbol $g$ to $ BMOA$, that is,
		\[
		\| g\|_{BMOA} \colonequals |g(0)| + \sup \{\| g\circ\varphi _w\|_{H^2}:\ w\in \mathbb{D}\}<\infty,
		\]
		where $\varphi _w(z) {\colonequals} \frac{w-z}{1-\bar{w}z}$ is the Möbius automorphism  of the unit  disc. Recall also that $I_g$ is compact on $H^2$ if and only if $g\in VMOA$, that is,
		\[
		\displaystyle \lim_{|w| \to 1^-}\| g\circ\varphi _w\|_{H^2}=0.
		\]
		The membership of $g$ to $BMOA$ or to $VMOA$ can be expressed using Carleson windows and the measure $d\nu_g(z) \colonequals |g'(z)|^2 (1-|z|^2)dm(z)$. Namely, it is proved in \cite{AS1} that
		\[
		g \in BMOA\quad \iff \frac{\nu_g(W(I_{n,k}))}{ |I_{n,k} |}=O(1)
		\]
		and 
		\[
		g \in VMOA \quad  \iff \displaystyle \lim_{n\to \infty} \sup _{1\leqs k \leqs 2^n}\frac{\nu_g(W(I_{n,k}))}{| I_{n,k}|}=0.
		\]
		
		\noindent	For $g \in VMOA$, denote by 
		{\[
			(\tau _n(\nu_g))_{n\geqs 1}\colonequals\left (\frac{\nu_g(W(I_{n}))}{{| I_n| }}  \right )_{n\geqs 1}
			\]
			the nonincreasing rearrangement of  the sequence $\left (\frac{\nu_g(W(I_{n,k}))}{{| I_{n,k}| }}  \right )_{n,k}$.}
		\\
		
		The main result in the case of the Hardy space is the following theorem.	
		
		
		\begin{theo}\label{Hardy}
			Let $g\in VMOA$. Then
			
			\begin{enumerate}[label=$(\roman*)$,leftmargin=* ,parsep=0cm,itemsep=0cm,topsep=0cm]
				\item\label{i} There exists an absolute integer $p\geqs 1$ such that 
				\[
				s_{p n}(I_g) \lesssim  \sqrt{\tau_n (\nu_g)}, \quad  n\geqs 1.
				\]
				\item \label{ii} There exists  an absolute constant $B>0$ such that
				\[
				\frac{1}{B} \displaystyle \sum _{j=1}^n 
				{\tau_j(\nu_g) } \leqs  \displaystyle \sum _{j=1}^n s_j^2(I_g),\quad n\geqs 1.
				\]
			\end{enumerate}
		\end{theo}
		
		Under some regularity conditions, we obtain a description of the asymptotic behavior of the singular values of $I_g$.
		\\ \\
		\noindent Let $\gamma >0$, a positive nonincreasing sequence $(u_n)_n$ is said to belong to $\cR_\gamma$ if there exists a nonincreasing sequence $(x_n)_n$ such that $n^\gamma x_n$ increases to infinity and $u_n \asymp x_n$. If, in addition, there exists $\alpha \in (0,\gamma)$ such that $n^\alpha x_n$ decreases, $(u_n)_n$ is said to belong to $\cR _{\gamma, \alpha}$.
		
		\begin{theo}\label{B}
			Let $g\in VMOA$. The following statements hold. \\
			\begin{enumerate}[label=$(\roman*)$,leftmargin=* ,parsep=0cm,itemsep=0cm,topsep=0cm]
				\item \label{2i}  Let $\gamma \in (0,1/2)$. If  {$(s_n(I_g))_n \in \cR_\gamma$ or $(\tau _n(\nu_g))_n \in \cR_{2\gamma}$}. Then
				\[
				s_n(I_g) \asymp\sqrt{ \tau _n(\nu_g)}.
				\]
				\item\label{3i} Let $\gamma \in (0,1)$ and let $\alpha \in (0,\gamma)$. If  {$(s_n(I_g))_n \in \cR_{\gamma,\alpha}$ or $(\tau _n(\nu_g))_n \in \cR_{2\gamma,2\alpha}$}. Then
				\[
				s_n(I_g) \asymp\sqrt{ \tau _n(\nu_g)}.
				\]
			\end{enumerate}
		\end{theo}
		\noindent The following corollary is a direct consequence of Theorem \ref{B}.
		\begin{cor}  Let $g \in VMOA$. Let $\alpha \in [0,1[$ and  $\beta \geqs 0$ such that $ \alpha >0$ or $\beta >0$ . We have
			\[
			s_n(I_g) \asymp \frac{1}{n^\alpha \log^\beta (n+1)} \quad \iff \tau_n(\nu_g) \asymp \frac{1}{n^{2\alpha} \log^{2\beta} (n+1)}.
			\]
		\end{cor}
		
		
		
		The paper is organized as follows. In Section \ref{sec2}, we prove Theorem \ref{Bergman}. Section \ref{sec3} is devoted to the proofs of the main  results in the Hardy space.
		\section{Bergman spaces}
		\label{sec2}
		We start this section by recalling the definitions of sampling and interpolating sets for Bergman spaces which play an important role in this paper.\\
		
		Let $Z = (z_n)_{n\geqs 1}$ be a sequence of points in $\D$.\\
		\begin{enumerate}[label=$\bullet$]
			\item  $Z$  is said to be separated if $\displaystyle \inf_{i\neq j}  \rho(z_i,z_j) > 0$, where $\rho(z,w) \colonequals   \left|\frac{z-w}{1-\bar{z}w}\right|$
			is the pseudohyperbolic metric on $\D$.
			
			\item  $Z$ is called a sampling set for $\cA^2_\alpha$ if 
			\[
			\|f\|^2_{\cA^2_\alpha} \asymp \sum_{n\geqs 1} |f(z_n)|^2(1-|z_n|^2)^{2+\alpha},\quad f \in \cA^2_\alpha.
			\]
			\item  $Z$  is called an interpolating set for $\cA^2_\alpha$ if for every sequence $(\lambda_n)_{n \geqs 1}$ satisfying 
			\[
			\displaystyle \sum_{n\geqs 1} |\lambda_n|^2(1-|z_n|^2)^{2+\alpha}<\infty,
			\]
			there exists $f$ in $\cA^2_\alpha$ such that $f(z_n)=\lambda_n$ for all $n\geqs 1.$ 
		\end{enumerate}
		The closed graph theorem shows that if $Z$ is an
		interpolating set for $\cA^2_\alpha$, then there exists a constant $C$ such that the interpolation problem can be
		solved by a function satisfying
		\[
		\|f\|_{\cA^2_\alpha} \leqs C \left(\sum_{n\geqs 1} |\lambda_n|^2(1-|z_n|^2)^{2+\alpha}\right)^{1/2}.
		\]
		The smallest such constant $C$  is called the constant of interpolation and we denoted it by $M_\alpha(Z)$. We draw attention to the fact that every interpolating set for $\cA^2_\alpha$ is separated. Interpolating and sampling sets  for Bergman spaces have been studied by many authors, we refer the reader to \cite{Seip1,Seip2} for more details.\\
		
		Let $\mu$ be a positive Borel measure on $\D$. Define the operator $T_\mu$ by
		\[ 
		T_\mu f(z)=\int_\mathbb{D} f(w)K^\alpha(z,w)(1-|w|^2)^\alpha d\mu(w),\  f\in \cA^2_\alpha,
		\]
		where $K^\alpha(z,w)=(1-z\bar{w})^{-2-\alpha}$ is the reproducing kernel of $\cA^2_\alpha.$ A simple computation yields
		\[
		\langle T_\mu f , f\rangle_{\cA^2_\alpha} = \displaystyle \int _\mathbb{D}|f(z)|^2 (1-|z|^2)^\alpha d\mu (z).
		\]
		
		\noindent These operators are called Toeplitz operators on Bergman spaces. Their boundedness was considered by Hasting \cite{Has}. Later, Luecking obtained necessary and sufficient condition on $\mu$ for which  $T_\mu$ belongs to a Schatten class $\mathcal{S}_p(\cA^2_\alpha)$, with $p>0$ (see \cite{Lue1}). Recently, the asymptotic behavior of the singular values of such operators has been studied on a large class of weighted Bergman spaces (see \cite{APP, EE}).\\ \\
		In the specific case of discrete measures, one can obtain explicit lower and upper bounds for the singular values of the associated Toeplitz operators. We have the following two lemmas.
		\begin{lem}\label{separation}
			Let $Z=(z_n)_{n\geqs 1}$ be a separated sequence and let $(c_n)_{n\geqs 1}$ be a sequence of positive numbers such that  $\left(c_n(1-|z_n|^2)^{-2}\right)_{n\geqs 1}$  decreases to zero. For $\mu= \displaystyle \sum_{n\geqs 1} c_n\delta_{z_n}$, the operator $T_\mu$ is compact on $\cA^2_\alpha$ and 
			\[
			s_n(T_\mu) \lesssim \dfrac{c_n}{(1-|z_n|^2)^2}.
			\]
		\end{lem}
		
		\begin{proof}
			First, recall \cite{GGK} that  for a compact operator $T$ on a Hilbert space $H$,
			\[
			s_n(T)  =  \inf\left\{\|T-R\| : R \colon H\to H\ \text{is of rank}\ <n \right\}.
			\]
			For  $\mu_n = \displaystyle \sum_{j=1}^{n-1} c_j\delta_{z_j}$, the associated operator $T_{\mu_n}$ is of rank $n-1$. Thus, it suffices to show that 
			$$\|T_\mu-T_{\mu_n}\| \lesssim \dfrac{c_n}{(1-|z_n|^2)^2}.$$
			Let $f \in \mathcal{A}^2_\alpha$. We have
			
			\begin{eqnarray*}
				\langle \left(T_\mu-T_{\mu_n}\right)f,f\rangle_{\cA^2_\alpha} & = & \displaystyle \sum_{j\geqs n} c_j (1-|z_j|^2)^\alpha|f(z_j)|^2\\ 
				& \leqs &  \frac{c_n}{(1-|z_n|^2)^2} \displaystyle \sum_{j\geqs n} (1-|z_j|^2)^{2+\alpha} |f(z_j)|^2.
			\end{eqnarray*}
			Moreover, since the sequence $Z$ is separated, there exists a small constant $\delta>0$ such that the discs $\cD_\delta(z_j)=\{z\in \D : |z-z_j|<\delta(1-|z_j|)\}$ are pairwise disjoint. 
			By subharmonicity, we get
			\[
			|f(z_j)|^2(1-|z_j|^2)^{2+\alpha} \lesssim \int_{\cD_\delta(z_j)} |f(z)|^2(1-|z|^2)^\alpha\ dm(z).
			\]
			Therefore
			\begin{eqnarray*}
				\langle \left(T_\mu-T_{\mu_n}\right)f,f\rangle_{\cA^2_\alpha} & \lesssim & \frac{c_n}{(1-|z_n|^2)^2} \sum_{j\geqs n} \displaystyle\int_{\cD_\delta(z_j)} |f(z)|^2(1-|z|^2)^\alpha\ dm(z)\\
				& \lesssim & \frac{c_n}{(1-|z_n|^2)^2}\int_{\D} |f(z)|^2(1-|z|^2)^\alpha\ dm(z)\\
				& = &  \frac{c_n}{(1-|z_n|^2)^2}\|f\| ^2_\alpha.
			\end{eqnarray*}
			Since the operator $T_\mu-T_{\mu_n}$ is positive, we obtain 
			\[
			\| T_\mu -T_{\mu _n}\| \lesssim  \frac{c_n}{(1-|z_n|^2)^2}.
			\]
		\end{proof}
		
		\begin{lem}\label{interpolation}
			Let $Z=(z_n)_{n\geqs 1}$ be an interpolating set for $\cA^2_\alpha$ and let $(c_n)_{n\geqs 1}$ be a sequence of positive numbers such that  $\left(c_n(1-|z_n|^2)^{-2}\right)_{n\geqs 1}$  decreases to zero. For $\mu= \displaystyle \sum_{n\geqs 1} c_n\delta_{z_n}$, the operator $T_\mu$ is compact on $\cA^2_\alpha$ and  
			\[
			s_n(T_\mu) \asymp \frac{c_n}{(1-|z_n|^2)^2}.
			\]
		\end{lem}
		
		\begin{proof}
			The upper bound result follows from Lemma \ref{separation}.  To obtain the lower bound, we use an alternative characterization {\cite[p. 49]{GGK}} of the singular values:
			\[
			s_n(T_\mu) =  \underset{E,\dim E=n}{\sup}\ \left(\underset{f\in E\setminus\{0\}}{\inf}\ \frac{\|T_\mu f\|_{\cA^2_\alpha}}{\|f\|_{\cA^2_\alpha}} \right).
			\] 
			Set
			$\cK_Z = span\{k_{z_j}^\alpha : j\geqs 1\}$, where $k_{z_j}^\alpha$ denotes the normalized reproducing kernel of $\cA^2_\alpha$. We choose as subspace $E$ the orthogonal complement in $\cK_Z$ of the linear span of  $\{k_{z_j}^\alpha : j > n\}$. Since $Z$ is an interpolating set for $\cA^2_\alpha$, the family $\{k_{z_j}^\alpha : j\geqs 1\}$ is a Riesz basis of $\cK_Z$ and then the vector space $E$ is of  dimension $n$.  
			Let $f\in E $, we have
			\begin{eqnarray*}
				\|f\|_{\cA^2_\alpha}^2\ \leqs\ [M_\alpha(Z)]^{2}\sum_{j\geqs 1} |\langle f,k_{z_j}^\alpha\rangle_{\cA^2_\alpha}|^2\ =\ [M_\alpha(Z)]^{2}\sum_{j=1}^n |\langle f,k_{z_j}^\alpha\rangle_{\cA^2_\alpha} |^2.
			\end{eqnarray*}
			Furthermore
			\begin{eqnarray*}
				\langle T_\mu f,f \rangle_{\cA^2_\alpha} & = & \sum_{j\geqs 1} c_j (1-|z_j|^2)^\alpha \left| f(z_j) \right|^2 \\
				& = & \sum_{j=1}^n c_j (1-|z_j|^2)^\alpha \left| f(z_j) \right|^2 \\
				& \geqs & \frac{c_n}{(1-|z_n|^2)^2} \sum_{j=1}^n |\langle f,{k_{z_j}^\alpha}\rangle_{\cA^2_\alpha}|^2,
			\end{eqnarray*}
			which finally leads to
			\[
			s_n(T_\mu)\gtrsim \left[M_\alpha(Z)\right]^{-2} \dfrac{c_n}{(1-|z_n|^2)^2}.
			\]
		\end{proof} 
		
		\begin{proof}[\textbf{Proof of Theorem \ref{Bergman}}]
			We start with the the upper bound of the singular values of $I_g$.  Let $Z=(z_n)_{n\geqs 1}$ be a  sampling set  for $\cA^2_{\alpha+2}$ such that each $R(I_n)$ contains $N$ points of $Z$, where $N$ is a large integer (see \cite{Lue2}). We denote these points by $(z_{n,i})_{n\geqs 1}$, where $i=1,\ldots,N$. 
			Using the fact that $Z$ is a sampling set for $\cA^2_{\alpha +2}$ with the subharmoncity of $|g'|^2$, we get
			\begin{eqnarray*}
				\int_{\D}|f(z)g'(z)|^2(1-|z|^2)^{2+\alpha}dm(z)& \asymp & \sum_{n\geqs 1}|f(z_{n})g'(z_{n})|^2(1-|z _{n}|^2)^{4+\alpha}\\
				& \lesssim & \sum_n  |f(z_{n})|^2(1-|z _{n}|^2)^\alpha \int _{R(I_n)}|g'(z)|^2(1-|z|^2)^2dm(z)\\
				& = & \sum_{n,i}  |f(z_{n,i})|^2(1-|z _{n,i}|^2)^\alpha \int _{R(I_n)}|g'(z)|^2(1-|z|^2)^2dm(z)	\\
				& = & \langle \displaystyle \sum_{i=1}^N T_{\mu_i} f,f\rangle_{\cA^2_\alpha},
			\end{eqnarray*}
			where $\mu_i = \displaystyle \sum _{n\geqs 1}m_g({R(I_n)})\delta _{z_{n,i}}.$ This implies
			\[
			I_g^*I_g  \lesssim \displaystyle   \sum_{i=1}^N  T_{\mu_i}.
			\] 
			{We use a known inequality of the singular values  from \cite[p. 30]{GKG}. If $A$ and $B$ are two compact operators, then
				\begin{equation}\label{eq(2.1)}
				s_{m+n-1}(A+B)\leqs s_m(A)+s_n(B), \quad m,n=1,2,\ldots.
				\end{equation}
				Thus we obtain	
			}	 
			\begin{eqnarray*}
				s_{Nn}(I_g^*I_g)  \lesssim  \displaystyle \sum_{i=1}^N s_n(T_{\mu_i}) \leqs  N  \max_{1\leqs i\leqs N} s_n(T_{\mu_i}).
			\end{eqnarray*}
			This, together with Lemma \ref{separation}, gives
			\[
			s_{Nn}^2(I_g)\lesssim N   \max_{1\leqs i\leqs N} s_n(T_{\mu_i}) \lesssim \frac{m_g({R(I_n)})}{(1-|z_n|^2)^2}.
			\]
			Now, we turn to prove the lower bound result. Choose a point $\xi_n$ in $R(I_n)$ so that $|g'(z)| \leqs 2|g'(\xi_n)|$ for all $z \in R(I_n)$.  Then
			\begin{eqnarray}\label{2}
			m_g(R(I_n))= \int _{R(I_n)}|g'(z)|^2 (1-|z|^2)^2dm(z) \leqs |g'(\xi_n)|^2 (1-|\xi_n|^2)^4.
			\end{eqnarray}
			Note that  $(\xi_n)_{n\geqs 1}$ can be written as a finite union of separated sets. Since every separated set itself can be expressed as a finite union of interpolating sets (see \cite{Seip2}), there exist  $\Lambda_1, \Lambda_2,\cdots ,\Lambda_M$ such that $\{\xi_n, \ n\in \Lambda _i\}$ is an interpolating sequence for  $\cA^2_{\alpha}$. Therefore 
			
			
			\begin{eqnarray*}
				\int_\D |f(z)g'(z)|^2(1-|z|^2)^{\alpha+2}dm(z)
				& \gtrsim &  \sum_{i=1}^M  \sum_{n\in \Lambda _i} |f(\xi_{n})|^2|g'(\xi_{n})|^2(1-|\xi_n|^2)^{\alpha+4}\\
				& \geqs &  \sum_{i=1}^M  \sum_{n\in \Lambda _i}|f(\xi_{n})|^2(1-|\xi_{n}|^2)^\alpha m_g(R(I_n))\\
				& = & \langle \sum_{i=1}^M T_{\mu_i} f,f\rangle_{\cA^2_{\alpha}},
			\end{eqnarray*}	
			where $\mu_i=\displaystyle \sum_{n\in \Lambda _i} m_g({R(I_n)})\delta_{\xi_{n}}$.
			The first inequality follows from the fact that $(\xi_n)_{n\geqs 1}$ is a finite union of separated sets. Therefore
			\[
			I_g^*I_g\gtrsim \sum_{i=1}^M T_{\mu_i} \geqs  T_{\mu_k},\quad k\in \{1,..,M\}.
			\] 
			This implies that 
			\[
			s_n(T_{\mu_k}) \lesssim s^2_n(I_g),\quad k\in \{1,..,M\}.
			\]
			Finally, using the fact that each {set $\{\xi_n, \ n\in \Lambda _i\}$}, for $i=1,\ldots,M$, is an interpolating set for ${\cA^2_{\alpha}}$ along with Lemma \ref{interpolation}, we get
			\[
			s_n(T_{\mu_k}) \gtrsim a_n(\mu_k),
			\]
			where $a_n(\mu _k)$ is the nonincreasing rearrangement of the sequence of the sequence $(\frac{\mu _k({R(I_n)})}{m({R(I_n)})})_n$.
			Consequently, there exists an integer $p$ such that $s_{[\frac{n}{p}]}^2(I_g)\gtrsim \frac{m_g({R(I_n)})}{(1-|z_n|^2)^2}.$ 
		\end{proof}
		\section{The Hardy space $H^2$}\label {hardy}
		\label{sec3}
		This section is devoted to prove Theorems \ref{Hardy} and \ref{B}. We start by recalling the connection between integration operators acting on the Hardy space $H^2$ and some standard operators.\\\\
		Let $\mu$ be a positive Borel measure on $\mathbb{D}$, and let $J_\mu$ denote  the embedding operator from $H^2$ into $L^2(\mu)$. Carleson proved in \cite{Car} that $J_\mu$  is bounded if and only if $\mu(W(I))/|I |$ is uniformly bounded for all arcs $I\subset\mathbb{T}$. This result can be expressed in terms of the family of the dyadic arcs. 
		Indeed, one can see that $J_\mu$ is bounded if and only if there exists $C>0$ such that 
		\[
		\mu ( W(I_{n,k}) ) \leqs C | I_{n,k}|, \quad \quad  n\geqs 1\ \mbox{and}\ 1\leqs k \leqs 2^n.
		\]
		Note also that $J_\mu$ is compact if and only if 
		\begin{eqnarray}\label{compact}
		\displaystyle \lim _{n\to \infty} \displaystyle \sup _{1\leqs k\leqs 2^n} \frac{\mu ( W(I_{n,k}) )}{| I_{n,k} |}= 0.
		\end{eqnarray}
		We introduce the operator 
		\[
		S_{\mu}f (z) = \displaystyle \int _\mathbb{D}f(w)K(z,w)d\mu(w),\quad f\in H^2,
		\]
		where $K(z,w)= (1-z\bar{w})^{-1}$ is the reproducing kernel of $H^2$. 
		A straightforward computation shows that
		\[
		\langle S_\mu f , f\rangle_{H^2} = \displaystyle \int _\mathbb{D}|f(z)|^2d\mu (z)= \| J_\mu f\|^2_{L^2(\mu)}.
		\]
		This implies that $S_\mu$ is bounded, compact, or belongs to a Schatten  class $\cS_{p}(H^2)$ if and only if $J_\mu$ is. In \cite{Lue1}, Luecking proved that $S_\mu \in \cS_p(H^2)$, for $p>0$, if and only if 
		\[
		\displaystyle \sum _{n,k} \left (\frac{\mu (R(I_{n,k}))}{| I_{n,k} |}\right )^{p}<\infty.
		\]
		It is worth noting that Luecking's characterization can also be expressed  in terms of Carleson  windows $W(I_{n,k})$ (see \cite{LLQRJFA}). However, this is no longer the case for Carleson boundedness criterion (see \cite{PR}).\\
		\\
		Now, we fix and recall some notations. Let $\mu$  be a positive Borel measure on $\D$ which satisfies condition (\ref {compact}), and let $(I_n(\mu))_{n\geqs 1}$ be an enumeration of  $(I_{n,k})_{n,k}$ such that the sequence $\left(\dfrac{\mu(W(I_n(\mu)))}{| I_n(\mu)|}\right)_{n\geqs 1}$ is nonincreasing. According to the notation used in the introduction, we define
		\[
		\tau _n(\mu) \colonequals \frac{\mu(W(I_n(\mu))}{| I_n(\mu)|}.
		\]

		\noindent The following lemma is the key to prove the first assertion of Theorem \ref{Hardy}.
		\begin{lem}\label{KL}
			Let $\mu$ be a positive Borel measure on $\mathbb{D}$ such that $S_\mu$ is compact on $H^2$. Let $I_n = I_n(\mu)$, and let $(z_n)_{n\geqs 1} $ be a sequence such that $z _n \in R(I_n)$. Fix an index $n\geqs 1$, and set $\mu_n =\displaystyle \sum _{k\geqs n}\mu (R(I_k))\delta _{z _k}$. We have  
			\begin{enumerate}
				\item\label{3}  
				$
				\displaystyle \sup _{j\geqs 1}\frac{\mu_n(W(I_j))}{| I_j|} \leqs \tau_n(\mu).
				$\\
				\item\label{4} $s_n(S_{\mu _1})\ \lesssim \ \tau_n(\mu).$\\
			\end{enumerate}
		\end{lem}	
		
		\begin{proof}
			We begin with part \eqref{3}. First, note that if $n=1$, then $\tau _j(\mu_1)= \tau _j(\mu)$ and the result is obvious.  Now, suppose that  $n\geqs 2$. \\ \\
			{For $j\geqs 1$, let $\Gamma_j=\lbrace k\geqs n \colon R(I_k) \subset W(I_j) \rbrace$. Then
				\begin{eqnarray*}
					\mu_n(W(I_j)) &=& \sum_{k \in \Gamma_j} \mu_n(R(I_k)) = \sum_{k \in \Gamma_j} \mu(R(I_k))  \leqs \mu(W(I_j)).
				\end{eqnarray*}	
				Thus for all $j\geqs n$, we have
			}
			\[
			\dfrac{\mu_n(W(I_j))}{| I_j |} \leqs \dfrac{\mu(W(I_j))}{| I_j |} \leqs \dfrac{\mu(W(I_n))}{| I_n |} .
			\]
			Let  $ j\in \lbrace 1,\ldots,n-1\rbrace$.  If we assume that
			\begin{eqnarray}\label{EL1}
			\dfrac{\mu_n(W(I_j))}{| I_j |} > \dfrac{\mu(W(I_n))}{| I_n |},
			\end{eqnarray}
			we will end up with a contradiction. Indeed, we will prove that there exists $j_1\in  \lbrace 1,\ldots,n-1\rbrace$ such that $W(I_{j_1})\varsubsetneq W(I_j)$ and $\dfrac{\mu (W(I_n))}{| I_n |} < \dfrac{\mu_n(W(I_{j_1}))}{| I_{j_1}|}$, which is absurd since one can repeat this operation indefinitely.\\
			
			Suppose that assumption \eqref{EL1} is satisfied, and write $W(I_j)= R(I_j)\cup W(I_{j_1})\cup W(I_{j'_1})$, where
			\[
			| I_{j_1} |= | I_{j'_1}|= | I_{j}|/2. 
			\]
			Without loss of generality, we assume that $\mu_n(W(I_{j_1}))\geqs \mu_n(W(I_{j'_1}))$. Then, since $\mu_n(R(I_j))=0$, we have
			\begin{eqnarray*}
				\dfrac{\mu_n(W(I_j))}{| I_j|} &=& \dfrac{\mu_n(W(I_{j_1}))+\mu_n(W(I_{j'_1}))}{| I_{j_1}|+| I_{j'_1}|}\leqs \dfrac{\mu_n(W(I_{j_1}))}{| I_{j_1}|}.
			\end{eqnarray*}
			This implies  $\dfrac{\mu (W(I_n))}{| I_n|} < \dfrac{\mu_n(W(I_{j_1}))}{| I_{j_1}|}$ and then $j_1\in  \lbrace 1,\ldots,n-1\rbrace$. The proof of \eqref{3} is complete.\\
			
			To prove \eqref{4}, recall that $s_n(S_{\mu_1}) =  \inf\left\{\left\| S_{\mu_1}- R\right\|\ :\ \text{rank}\  R<n \right\} $. For $R = S_{\mu_1 -\mu_n}$, we have 
			\[
			s_n(S_{\mu_1}) \leqs \| S_{\mu_1} - S_{\mu_1-\mu_n}\| = \|S_{\mu _n}\| \lesssim  \displaystyle \sup _{j\geqs 1}\frac{\mu_n(W(I_j))}{| I_j|}.
			\]
			The second inequality comes from the embedding Carleson's Theorem. Using the first part of the lemma, we obtain $  s_n(S_{\mu_1}) \lesssim \tau _n(\mu)$.
		\end{proof}
		The proof of the second assertion of Theorem \ref{Hardy} requires the following lemma. We include the proof for completeness.
		\begin{lem}\label{0}
			Let $H$ and $K$ be two Hilbert spaces and  $T: H\to K$  a compact operator. Suppose that there exist $(u_n)_{n\geqs 1}\subset H$ and $(v_n)_{n\geqs 1}\subset K$ such that 
			\[
			\displaystyle \sum _n |\langle u_n, f \rangle_H |^2\leqs C\|f\|_H^2 \ \text{and} \ \displaystyle \sum _n |\langle v_n, g\rangle_K |^2\leqs C\|g\|_K^2,\quad f\in H,\  g\in K,
			\]
			for some constant $C>0$. Then  we have 
			\[
			\displaystyle \sum _{j=1}^n \left|\langle Tu_j, v_j\rangle_{K} \right|\leqs C^2 \displaystyle \sum _{j=1}^n s_j(T).
			\]
		\end{lem}
		
		\begin{proof}
			Let $(e_n)_n$ and $(f_n)_n$ be two orthonormal bases of $H$ and $K$, respectively. Let $A$ and $B$ be the operators given by
			\[
			Af=\displaystyle\sum_n\langle u_n,f\rangle_H e_n \quad \text{and}\quad   Bg=\displaystyle\sum_n\langle v_n,g\rangle_K f_n.
			\]
			By assumptions, $A$ and $B$ are bounded and $\|A\|, \|B\| \leqs C$. Note that 
			the adjoint operators $A^*$ and $B^*$ satisfy $A^*e_n=u_n$ and $B^*f_n=v_n$. By Theorem 3.5 of \cite{GGK}, we have 
			\[
			\displaystyle \sum_{j=1}^{n}\left|\langle BTA^* e_j,f_j\rangle \right|  \leqs \displaystyle \sum_{j=1}^{n} s_j(BTA^*).
			\]
			Thus
			\begin{eqnarray*}
				\displaystyle \sum_{j=1}^{n}\left|\langle Tu_j,v_j \rangle \right| &=& \displaystyle \sum_{j=1}^{n}\left|\langle TA^* e_j,B^*f_j\rangle \right |\\
				& = & \displaystyle \sum_{j=1}^{n}\left|\langle BTA^* e_j,f_j\rangle \right|\\
				&\leqs& \displaystyle \sum_{j=1}^{n} s_j(BTA^*)\\
				&\leqs& \|B\|\|A^*\|\displaystyle \sum_{j=1}^{n}s_j(T)\leqs C^2 \displaystyle \sum_{j=1}^{n}s_j(T).
			\end{eqnarray*} 
			
		\end{proof}
		
		\noindent The classical Littlewood-Paley identity \cite{Gar} states that
		\[
		\| f -f(0)\|^2 _{H^2}=  \displaystyle \int_{\mathbb{D}}|f'(z)|^2\log (1/|z|^2)dm(z).
		\]
		This  gives rise to the following formula which relates the operators $I_g$ and $S_{\nu _g}$. We have
		\begin{eqnarray}\label{H-L}
		\langle I^*_gI_g f, f \rangle_{H^2}  = \displaystyle \int _{\mathbb{D}}|f(z)|^2|g'(z)|^2\log (1/|z|^2)dm(z)
		\asymp \langle S_{\nu_g}f,f \rangle_{H^2}.
		\end{eqnarray}
		
		\begin{proof}[\textbf{Proof of Theorem \ref{Hardy}}]
			It is immediate, from formula (\ref{H-L}), that if $f\in H^2$ then $fg'\in \cA^2_1$. We have seen in the proof of the upper bound result from Theorem \ref{Bergman} that
			there exist a sampling sequence $Z=(z_n)_{n\geqs 1}$ for $\cA^2_1$ and an integer $N$ such that each $R(I_n)$ contains  $N$ points of $Z$ denoted by $(z_{n,i})_{n\geqs 1}$, where $i=1,\ldots,N$. Then
			\begin{eqnarray*}
				\displaystyle \int _{\mathbb{D}}|f(z)g'(z)|^2(1-|z|^2)dm(z)& \asymp& \displaystyle \sum _{n\geqs 1} |f(z_{n})g'(z _{n})|^2(1-|z _{n}|^2)^3\\
				& \lesssim &  \displaystyle \sum _{n\geqs 1} |f(z_{n})|^2 \displaystyle \int _{R(I_n)}|g'(z)|^2(1-|z|^2)dm(z)\\
				& = &  \displaystyle \sum _{n\geqs 1} \displaystyle \sum _{i=1}^N|f(z_{n,i})|^2 \displaystyle \int _{R(I_n)}|g'(z)|^2(1-|z|^2)dm(z)\\
				& = & \langle \displaystyle \sum _{i=1}^N S_{\nu_i} f,f\rangle_{H^2},
			\end{eqnarray*}
			where $\nu_i = \displaystyle \sum _{n\geqs 1}\nu_g(R(I_n))\delta _{z_{n,i}}.$ {Using inequality \eqref{eq(2.1)}, we obtain}
			\[
			s_{Nn}^2(I_g)\asymp s_{Nn}(S_{\nu_g}) \lesssim  \max_{1\leqs i \leqs N} s_n(S_{\nu_i}).
			\]
			{By} Lemma \ref{KL}, we deduce  $s_{Nn}^2(I_g) \lesssim \tau _n(\nu_g)$ as desired and the proof of assertion \ref{i} is complete.\\
			
			\noindent To prove assertion \ref{ii}, we use Lemma \ref{0}. Let $\xi_n$ be the center of  $R(I_n)$ and set
			\[
			u_n(z)=v_n(z)=\frac{(1-|\xi_n|^2)^{3/2}}{(1-\bar{\xi}_n z)^2}.
			\]
			For $f \in H^2$, we have
			\begin{eqnarray*}
				|\langle u _n, f\rangle _{H^2}|^2 &  \lesssim &  | u _n (0)f(0)|^2+\left| \displaystyle \int _{\mathbb{D}}\frac{(1-|\xi_n|^2)^{3/2}}{(1-\bar{\xi}_n z)^3} \overline{f'(z)} (1-|z|^2)dm (z)\right|^2\\
				& = & (1-|\xi_n|^2)^3 |f(0)|^2+|\langle k_{\xi_n}^1,f' \rangle_{\mathcal{A}^2_1} |^2.
			\end{eqnarray*}
			Thus 
			\begin{eqnarray*}
				\displaystyle \sum _n |\langle u _n, f\rangle _{H^2}|^2 &  \lesssim & \displaystyle \sum _n (1-|\xi_n|^2)^3 |f(0)|^2+ \displaystyle \sum _n |\langle k_{\xi_n}^1,f' \rangle_{\mathcal{A}^2_1} |^2\\
				& \lesssim & \displaystyle \sum _n (1-|\xi_n|^2)^3 |f(0)|^2+\|f'\|^2_{\cA^2_1}\\
				& \lesssim & \|f\|^2_{H^2}.
			\end{eqnarray*}
			The second inequality follows from the fact that $(\xi _n)_n$ is a separated sequence.
			Then, by Lemma \ref{0} applied to $S_{\nu _g}$, there exists an absolute constant $c>0$ such that
			\[
			\displaystyle c \sum _{j=1}^n |\langle S_{\nu_g} u_j, u_j\rangle _{H^2}| \leqs \displaystyle \sum _{j=1}^n s_j(S_{\nu_g}).
			\]
			Furthermore
			\begin{eqnarray*}
				\langle S_{\nu_g} u_n,u_n\rangle_{H^2} & = & \displaystyle \int_\mathbb{D} \dfrac{(1-|a_n|^2)^3}{|1-\bar{a}_n z|^4}|g'(z)|^2(1-|z|^2) dm(z) \\
				& \geqs & \displaystyle \int_{W(I_n)}\dfrac{(1-|a_n|^2)^3}{|1-\bar{a}_n z|^4}|g'(z)|^2(1-|z|^2) dm(z)\\
				& \gtrsim & \dfrac{1}{|I_n|} \displaystyle \int_{W(I_n)} |g'(z)|^2(1-|z|^2) dm(z).
			\end{eqnarray*}
			Combining these two facts with formula (\ref{H-L}), we get 
			\[
			\displaystyle  \sum _{j=1}^n  \frac{\nu_g(W(I_j))}{|I_j|} \lesssim  \displaystyle  \sum _{j=1}^n s_j(S_{\nu_g}) \lesssim   \displaystyle  \sum _{j=1}^n s^2_j({I_{g}}).
			\]
		\end{proof}
		
		
		We turn now to prove Theorem \ref{B}, we start with the proof of the first statement which is completely based on Theorem \ref{Hardy}.
		
		\begin{proof}[\textbf{Proof of assertion \ref{2i} of Theorem \ref{B}} ]
			Suppose that $(s_n(I_g))_n \in \cR _\gamma$ for some $\gamma \in (0,1/2)$. Then  there exists $(x_n)_n$ a nonincreasing sequence  of positive numbers such that ${n^{\gamma} x_n}$ increases to infinity and  $s_n(I_g) \asymp x_n$.  {For an integer $p\geqs 1$, we have  
				\begin{eqnarray*}
					p^{-\gamma}x_n \leqslant x_{pn} \leqslant x_n. 
				\end{eqnarray*}
			}
			By assertion \ref{i} of Theorem \ref{Hardy}, we have 
			\begin{eqnarray}\label{UB}
			{x^2_n  \asymp s_{np}^2(I_g)} \lesssim \tau_n(\nu_g).
			\end{eqnarray}	
			Moreover, Theorem \ref{Hardy} implies
			\begin{eqnarray*}\label{LB}
				n\tau_n(\nu_g) \leqs \displaystyle \sum_{j=1}^n \tau_j(\nu_g)  & \lesssim & \displaystyle  \sum_{j=1}^n s_j^2(I_g)\\
				&\lesssim & {\displaystyle  \sum_{j=1}^n x^2_j}  \\
				& \leqs & n^{2\gamma} x^2_n \displaystyle  \sum_{j=1}^n \dfrac{1}{j^{2\gamma}}. 
			\end{eqnarray*}	
			{The last inequality follows from the fact that $(j^{2\gamma}x_j)_j $ is increasing. Since $2\gamma <1$, we have
				\begin{eqnarray*}
					\sum_{j=1}^n \dfrac{1}{j^{2\gamma}} &\asymp& n^{1-2\gamma}.	
				\end{eqnarray*}	
				Thus $n\tau_n(\nu_g)	\lesssim n x^2_n.$}
			This along with inequality (\ref{UB}) leads to {$ x_n \asymp \tau_n(\nu_g)$.}\\ \\
			Now, suppose that  $(\tau_n(\nu_g))_n \in \cR _{2\gamma}$ and let $(x_n)_n$ be a nonincreasing sequence  of positive numbers such that $n^{2\gamma} x_n$ increases to infinity and  $\tau_n(\nu_g) \asymp x_n$. It follows at once from assertion \ref{i} of Theorem \ref{Hardy} that 
			{
				\begin{eqnarray*}
					s_n^2(I_g) \lesssim x_{[\frac{n}{p}]} \asymp x_n. 
				\end{eqnarray*}
			}
			Let  $\kappa$ be a large constant. Assertion \ref{ii} of Theorem \ref{Hardy} yields
			\begin{eqnarray*}
				\displaystyle \sum_{j=1}^{\kappa n} x_j & \lesssim & \displaystyle \sum_{j=1}^n s_j^2(I_g) +  \displaystyle \sum_{j=n+1}^{\kappa n} s_j^2(I_g) \\
				& \lesssim & \displaystyle \sum_{j=1}^n x_j + \kappa n s_n^2(I_g).
			\end{eqnarray*}
			{Using the same line of reasoning as in the previous part, we obtain $\displaystyle\sum_{j=1}^n x_j \lesssim n x_n.$ Therefore	
				\begin{eqnarray*}
					\displaystyle \sum_{j=1}^{\kappa n} x_j &\lesssim & n x_n + \kappa n s_n^2(I_g).	
				\end{eqnarray*}
				In contrast
				\begin{eqnarray*}
					\displaystyle \sum_{j=1}^{\kappa n} x_j  & \geqs  & \kappa n x_{\kappa n} \\ 
					&\geqs& \kappa^{1-2\gamma} n x_n.
				\end{eqnarray*}
				The first inequality follows from the fact  $(x_j)_j $ is nonincreasing and the second one uses the fact that $(j^{2\gamma}x_j)_j $ is increasing. 
				Thus
				\[
				\kappa^{1-2\gamma} x_n \lesssim x_n + \kappa s_n^2(I_g).
				\]
			}
			For  $\kappa$ large enough, we obtain the result.
		\end{proof}
		
		The proof of the second statement of Theorem \ref{B} requires some preparatory results. The following technical lemma is a generalization of Proposition 3.3 in \cite{LLQRJFA}.
		\begin{lem}\label{lemLLQP}
			Let $h: [0,\infty[\to [0,\infty[$ be an increasing function such that $h(0)=0$. Suppose that there exist $\varepsilon \in (0,1)$ and $p\geqs 1$ for which $t\to h(t^\varepsilon)$ is concave and $t\to h(t^p)$ is convex. Let $\mu $ be a positive Borel measure on $\mathbb{D}$. There exists $B>0$, which depends on $\varepsilon$ and $p$, such that  
			\[
			\displaystyle \sum _{j}h\left (\frac{\mu(W{(I_j)})}{| I_j |}\right ) \leqs \displaystyle \sum _{j}h\left (B\frac{\mu(R{(I_j)})}{| I_j|}\right ).
			\]
		\end{lem}
		\begin{proof}
			We use the argument given in \cite{LLQRJFA}. Let $\alpha = 1/\varepsilon$ and let $\beta =1/(1-\varepsilon)$ be the conjugate of $\alpha$. Write
			\[
			{W(I_{n,j})}=\bigcup_{l\geqs n}\bigcup_{k\in H_{l,n,j}}  {R(I_{l,k})},\quad n\geqs 1\ \mbox{and}\ 1\leqs j\leqs 2^n,
			\]
			where $H_{l,n,j}=\left\lbrace k\in\lbrace 1,\ldots,2^l \rbrace; \dfrac{j}{2^n}\leqs\dfrac{k}{2^l}<\dfrac{j+1}{2^n}\right\rbrace.$ Consider $a \in (1,2)$ such that $2< a^\beta$. It is proved in \cite{LLQRJFA} that 
			\[
			\mu({W(I_{n,j})})^\alpha \lesssim \displaystyle \sum _{l\geqs n; \ k\in H_{l,n,k}}a^{(l-n)\alpha }\mu({R(I_{l,k})})^\alpha.
			\]
			Recall that the function $h_{1/\alpha}(t) \colonequals h(t^{1/\alpha})$ is increasing and concave. Consequently,
			\begin{eqnarray*}
				h(2^n\mu({W(I_{n,j})}) )& =& h_{1/\alpha} \left ( (2^n\mu({W(I_{n,j})}))^\alpha \right )\\
				& \lesssim & \displaystyle \sum _{l\geqs n; \ k\in H_{l,n,k}}  h_{1/\alpha} \left (a^{(l-n)\alpha}2^{n\alpha}\mu({R(I_{l,k})})^\alpha \right )\\
				& = & \displaystyle \sum _{l\geqs n; \ k\in H_{l,n,k}}  h\left (\left (\frac{a}{2}\right )^{l-n}2^l\mu({R(I_{l,k})}) \right )\\
				& \leqs & \displaystyle \sum _{l\geqs n; \ k\in H_{l,n,k}}  \left (\frac{a}{2}\right ) ^{(l-n)/p}h(2^l\mu({R(I_{l,k})}) ).\\
			\end{eqnarray*}
			The last inequality follows from the fact $h(t^p)$ is convex. Then we obtain 
			\begin{eqnarray*}
				\sum_{n= 1}^\infty\sum_{j=0}^{2^n-1}h\left(2^n\mu({W(I_{n,j})})\right)&\lesssim &\sum_{l=1}^\infty\sum_{k=0}^{2^l-1}\left(\sum_{(n,j):l\geqs n,k\in H_{l,n,j}}\left (\frac{a}{2}\right ) ^{(l-n)/p}\right)h\left(2^l\mu({R(I_{l,k})})\right)\\
				&\leqs&C{(p,\varepsilon)}\sum_{l=1}^\infty\sum_{k=0}^{2^l-1}h\left(2^l\mu({R(I_{l,k})})\right).
			\end{eqnarray*}
		\end{proof}
		
		\begin{theo}\label {trace}
			Let $\mu$ be a positive Borel measure on $\mathbb{D}$ such that $J_\mu$ is compact. Let $h: [0,\infty[\to [0,\infty[$ be a convex  increasing function with $h(0)=0$ and such that $t\to h(t^\varepsilon)$ is concave for some $\varepsilon \in (0,1)$. Then there exists $B>0$ which depends only on $\varepsilon $ such that 
			\[
			\displaystyle \sum _{j}h\left (\frac{1}{B}\sqrt{\tau_j(\mu)}\right ) \leqs \displaystyle \sum _{j}h(s_j(J_\mu)).
			\]
		\end{theo}
		
		\begin{proof} The proof is based on Lemma \ref{0}. Let $\xi_n$ be the center of  ${R(I_n)}$, and set
			\[
			u_n(z)=\frac{(1-|\xi_n|^2)^{3/2}}{(1-\bar{\xi}_n z)^2}\ \mbox{and}\ v_n= c_n \theta_n \chi_{{R(I_n)}},
			\]
			where $\theta _n = \frac{u_n}{|u_n|}$ and $c_n= \frac{1}{\sqrt{\mu ({R(I_n)})}}$ if $\mu ({R(I_n)}) \neq 0$.\\
			We proved before 
			\[
			\displaystyle \sum _n |\langle u _n, f\rangle _{H^2}|^2  \lesssim \|f\|^2_{H^2}, \quad f \in H^2.
			\]
			By Cauchy-Schwarz inequality, we have 
			\[
			\displaystyle \sum _n |\langle v_n, v\rangle _{L^2(\mu)}|^2 \leqs \displaystyle \sum _n c^2_n\mu ({R(I_n)})\displaystyle \int_{{R(I_n)}}|v|^2d\mu =  \| v \| ^2_{L^2(\mu)},\quad v \in L^2(\mu).
			\]
			Note also that 
			\[
			\langle J_\mu u_n, v_n\rangle _{L^2(\mu)} = c_n \displaystyle \int _{{R(I_n)}}|u_n|d\mu \asymp \sqrt{\frac{\mu({R(I_n)})}{| I_n |}}.
			\]
			Applying Lemma \ref{0}, we get
			\[
			\displaystyle  \sum _{j=1}^n  \frac{1}{C}\sqrt{\frac{\mu(R(I_j))}{| I_j |}} \lesssim  \displaystyle  \sum _{j=1}^n s_j(J_\mu),
			\]
			where $C$ is a positive constant. 
			Then by Corollary 3.3 of \cite{GGK}, we obtain
			\[
			\displaystyle \sum _{j}h\left (\frac{1}{C}\sqrt{\frac{\mu(R{(I_j)})}{| I_j|}}\right ) \leqs \displaystyle \sum _{j}h(s_j(J_\mu)).
			\]
			The desired result follows from Lemma \ref{lemLLQP}.
		\end{proof}
		\begin{cor}\label {Ctrace}
			Let $g\in VMOA$. Let $h: [0,\infty[\to [0,\infty[$ be a convex  increasing function with $h(0)=0$ and such that $t\to h(t^\varepsilon)$ is concave for some $\varepsilon \in (0,1)$. Then there exists $B>0$ which depends only on $\varepsilon $ such that 
			\[
			\displaystyle \sum _{j}h\left (\frac{1}{B}\sqrt{\tau_j(\nu_g)}\right ) \leqs \displaystyle \sum _{j}h(s_j(I_g))).
			\]
		\end{cor}
		\begin{proof}
			It suffices to apply Theorem \ref{trace} and to remark that $s_n(I_g)= s_n(J_{\nu_g})$.
		\end{proof}
		
		The following lemma is proved in \cite{EE}.
		\begin{lem}\label{lemEE}
			Let $(a_n)_n$ and $(b_n)_n$ be two positive {nonincreasing} sequences. Suppose  that there exist $\beta _1 >1$ and $\beta _2>1 $ such that $(n^{\beta _1} b_n)_n$ is {nonincreasing} and $(n^{\beta _2}b_n)_n$ is increasing. Let $\beta >\beta _2$  and let $B>0$. If for each positive increasing concave function $h$ such that $h(t^\beta)$ is convex we have
			\[
			\displaystyle \sum _{n}h (\frac{1}{B}a_n) \leqs \displaystyle \sum _{n}h (b_n) \leqs \displaystyle \sum _{n}h(B a_n),
			\]
			then $a_n \asymp b_n$.
		\end{lem}
		\begin{proof}[\textbf{Proof of assertion \ref{3i}  of Theorem \ref{B}}] 
			Suppose that $(s_n(I_g))_n \in \cR_{\gamma, \alpha}$. Then there exists a {nonincreasing} sequence $(x_n)_n$ such that $(n^\gamma x_n)_n$ is increasing, $(n^\alpha x_n)_n$ is {nonincreasing}, and $x_n \asymp s_n(I_g)$. We have to prove that $x_n \asymp \sqrt{\tau _n(\nu_g)}$. 
			Let $\beta>0$ be such that $\beta \alpha >1$. Put $\beta _1= \beta \alpha$ and $\beta _2= \beta \gamma$. Let $h$ be a positive increasing concave function such that $h(0)= 0$ and {$h_\beta \colonequals h(t^\beta) $} is convex. The function $h_\beta$ satisfies the assumptions of Corollary \ref{Ctrace}, with $\varepsilon = 1/\beta$. Then there exists $B_1$ such that
			\[
			\displaystyle \sum _{j}h_\beta \left (\frac{1}{B_1}\sqrt{\tau _j(\nu_g)}\right ) \leqs \displaystyle \sum _{j}h_\beta (x_j).
			\]
			By the first assertion of Theorem \ref{Hardy}, we have
			\[
			x_n \asymp s_{p n}(I_g) \lesssim  \sqrt{\tau_n (\nu_g)}, \quad  n\geqs 1.
			\]
			Then there exists $B>0$ such that 
			\[
			\displaystyle \sum _{j}h_\beta\left (\frac{1}{B}\sqrt{\tau _j(\nu_g)}\right ) \leqs \displaystyle \sum _{j}h_\beta (x_j)\leqs \displaystyle \sum _{j}h_\beta\left (B\sqrt{\tau _j(\nu_g)}\right ).
			\]
			Lemma \ref{lemEE} yields $x_n^{\beta} \asymp \tau ^{\beta/2}_n (\nu_g)$, that is, $x_n \asymp \sqrt{\tau _n(\nu_g)}$.\\
			The result in the case $(\tau _n (\nu_g))_n \in \cR_{2\gamma, 2\alpha}$ can be proved in the same way.
		\end{proof}
		
		Using Theorems \ref{Hardy} and \ref{Ctrace} with  the same line of reasoning as in the proofs of Theorem 1.2 and Proposition 5.2 of \cite{BEMN}, one can obtain the following result. We omit the proof here.
		\begin{cor}
			Let $g\in VMOA$. We have\\
			\begin{enumerate}
				\item  $s_n(I_g)= o\left ( \frac{1}{n}\right ) \quad  \Longrightarrow \quad I_g =0.$
				\item $s_n(I_g)= O\left ( \frac{1}{n}\right ) \quad \iff \quad g' \in H^1.$
			\end{enumerate}
		\end{cor}
		
		

		\bibliographystyle{plain} 
		
	\end{document}